\documentclass[11pt,reqno]{amsart}
\usepackage{amsfonts,amsmath,amssymb,enumerate}
\usepackage{geometry}
\usepackage{amsthm}
\usepackage{mathrsfs}  
\usepackage{bm}
\usepackage{color}
\numberwithin{equation}{section}
\theoremstyle{plain}
\newtheorem{thm}{Theorem}[section]
\newtheorem{lem}[thm]{Lemma}
\newtheorem{prop}[thm]{Proposition}

\theoremstyle{definition}

\theoremstyle{remark}

\newcommand{\bs}{\boldsymbol}
\newcommand{\be}{\begin{equation}}    
\newcommand{\ee}{\end{equation}}    
\newcommand{\beu}{\begin{equation*}}    
\newcommand{\eeu}{\end{equation*}}    
\newcommand{\bea}{\begin{eqnarray}}    
\newcommand{\eea}{\end{eqnarray}}    
\newcommand{\beaa}{\begin{eqnarray*}}    
	\newcommand{\eeaa}{\end{eqnarray*}}    
\newcommand{\bmx}{\begin{bmatrix}}    
	\newcommand{\emx}{\end{bmatrix}}
\newcommand{\bpm}{\begin{pmatrix}}
	\newcommand{\epm}{\end{pmatrix}}
\newcommand{\Z}{\mathbb{Z}}
\newcommand{\C}{\mathbb{C}}
\newcommand{\gl}{\mathfrak{gl}}

\newcommand{\de}{\delta}
\newcommand{\wt}{\widetilde}
\newcommand{\ber}{{\rm Ber}\,}
\newcommand{\bers}{{\rm Ber}^{\bs s}\,}
\newcommand{\sgn}{ {\rm sgn}\,}
\newcommand{\End}{{\rm End}\,}

\newcommand{\si}{\sigma}

\newcommand{\la}{\lambda}
\newcommand{\La}{\Lambda}
\newcommand{\dd}{\partial}
\newcommand{\dv}{\partial_v}
\newcommand{\du}{\partial_u}

\newcommand{\BB}{\mathfrak{B}}
\newcommand{\cdet}{{\rm cdet}\,}

\newcommand{\sS}{\mathfrak{S}}
\newcommand{\diag}{{\rm diag}}

\author{Chenliang Huang} 
\address{Department of Mathematical Sciences, 402
	N. Blackford St, LD 270, IUPUI, Indianapolis, IN 46202, USA.  }
\email{ch30@iupui.edu}

\author{Evgeny Mukhin} 
\address{Department of Mathematical Sciences, 402
	N. Blackford St, LD 270, IUPUI, Indianapolis, IN 46202, USA.  }
\email{emukhin@iupui.edu}

\begin{document}
\title[Duality of supersymmetric Gaudin models] {The duality of $\gl_{m|n}$ and $\gl_k$ Gaudin models}
\begin{abstract}
We establish a duality of the non-periodic Gaudin model associated with superalgebra $\gl_{m|n}$ and the non-periodic Gaudin model associated with algebra $\gl_k$.

The Hamiltonians of the Gaudin models are given by expansions of a Berezinian of an $(m+n)\times(m+n)$ matrix in the case of $\gl_{m|n}$ 
and of a column determinant of a $k\times k$ matrix in the case of $\gl_k$. We obtain our results by proving Capelli type identities for both cases and comparing the results.
\end{abstract}
\maketitle
\section{Introduction}\label{sec:int}
Integrable models associated with finite-dimensional Lie superalgebras have been recently receiving the much deserved attention. While most of the work is done by physicists on the spin-chain side, the theory of the corresponding Gaudin models is also moving forward, see \cite{MR}, \cite{MVY}, \cite{HMVY}. 
The duality of various systems is another very important topic which always gets a lot of attention. In this paper we discuss the duality 
of the Gaudin model associated with supersymmetric $\gl_{m|n}$ to the Gaudin model associated with even $\gl_k$ acting on the same bosonic-fermionic space.

In the Lie algebra duality setting, the Lie superalgebras $\gl_{m|n}$ and $\gl_k$ both act on the algebra of supersymmetric polynomials $V$ generated by entries of the $(m+n)\times k$ matrix $(x_{i,a})$ where $x_{i,a}$ is even if and only if $i\leq m$. Then each row is identified with the vector representation of $\gl_k$ and each column with the vector representation of $\gl_{m|n}$. The two actions are extended to the action on the whole bosonic-fermionic space $V$ of supersymmetric polynomials  as differential operators, where they centralize each other, see Section \ref{classical duality}. We chose column evaluation parameters $z_1,\dots,z_k$ for $\gl_{m|n}$, row evaluation parameters $\La_1,\dots,\La_{m+n}$ for $\gl_{k}$ and upgrade the action to the current algebras $\gl_{m|n}[t]$ and  $\gl_{k}[t]$ in $V$ so that each row and each column becomes an evaluation module with the corresponding evaluation parameter.

It is well known that the commuting Hamiltonians of the $\gl_k$ Gaudin system are elements of $U\gl_{k}[t]$ given by the coefficients of the column determinant of the $k\times k$ matrix $G=\big(\delta_{a,b}(\dd_u-z_a)-e_{a,b}^{[k]}(u)\big)$, see \cite{T}, where we chose evaluation  parameters of columns $z_1,\dots,z_k$ to be the so called boundary parameters of the model. 

It is also known that the Hamiltonians of the $\gl_{m|n}$ Gaudin system are elements of $U\gl_{m|n}[t]$ given by the coefficients of the Berezinian of the $(m+n)\times(m+n)$ matrix $B=\big(\delta_{i,j}(\dd_v-\La_i)-e_{i,j}^{[m|n]}(v)\big)$, see \cite{MR}, \cite{MM}, and Section \ref{subsec:balgebra}. Note that we chose evaluation  parameters of rows $\La_1,\dots,\La_{m+n}$ to be the boundary parameters of the model.

The column determinant $\cdet G$ is a differential operator of order $k$ in variable $u$ whose coefficients are power series in $u^{-1}$.
The Berezinian $\ber B$ a pseudodifferential operator in $\dv^{-1}$ whose coefficients are power series in $v^{-1}$.
Our main result is that after multiplying by simple factors, coefficients of $v^r\dd_v^s$ and of $u^s\dd_u^r$ of the two expansion coincide as differential operators in $V$, see Theorem \ref{thm:duality between bethe algebras}. 

In order to prove our main result we establish two Capelli-like identities, see Propositions \ref{prop:capelli} and \ref{prop:ber for gl_m|n}, which give the normal ordered expansions of the $\cdet G $ and $\ber B$ acting in $V$. Because of the presence of fermions, those expansions have more terms than the original Capelli identity. However, the main feature is the same: the quantum corrections created by non-commutativity all cancel out and the result is the same as it would be in the supercommutative case.

The expansion of the $\cdet G$ is done by careful accounting of all terms and finding a way to cancel or collect the terms. For the Berezinian expansion we exploit a few tricks. Namely, we represent $\ber B$ as a Berezinian of a matrix of size $(m+n+k)\times (m+n+k)$ then interchange the rows and columns to reduce the computation to another column determinant. The key property which allows us to do it, is the super version of Manin property of the matrices with some additional property which we call "affine-like". The affine-like property guarantees the existence of various inverse matrices and the Manin property of those inverses, see Section \ref{sec:Ber}. In particular, we argue that for such matrices the Berezinian can be defined via quasi-determinants, similar to affine Manin matrices of standard parity treated in \cite{MR}.

\medskip

The spectrum of Gaudin Hamiltonians is found by the Bethe ansatz, see \cite{MTV1} for the even and  \cite{MVY} for the supersymmetric case. Since the two sets of Hamiltonians actually coincide in $V$, we have a correspondence between solution sets of two very different systems of the Bethe ansatz equations. Moreover, the eigenvectors of $\gl_k$ model are in a natural bijection with differential operators of order $k$ with quasipolynomial kernels, see \cite{MTV4}, while eigenvectors of $\gl_{m|n}$ model are conjecturally in a bijection with ratios of differential operators of orders $m$ and $n$, and appropriate superspaces of quasirational functions, cf. \cite{HMVY}.

The duality of the $\gl_n$ and $\gl_m$ systems was established in \cite{MTV5}. The corresponding map between spaces of polynomials is given by an appropriate Fourier transform  and it is also identified with the bispectrality property of the KP hierarchy, see \cite{MTV2}. It is important to understand this map in the supersymmetric case.

\medskip

We expect that the results of this paper can be extended to the most general duality 
of Gaudin models associated with $\gl_{m|n}$ and $\gl_{k|l}$. We also expect that a similar duality can be established 
in the Yangian,  see \cite{MTV3}, and the quantum setting.

The duality between $\gl_{1|1}$ and $\gl_2$ Gaudin models has appeared in \cite{BBK}.

\medskip

The paper is constructed as follows. In Section \ref{sec:preliminary} we establish our conventions and notation. We discuss Berezinians and their properties in Section \ref{sec:Ber}. Section \ref{sec:gl_{m|n} bethe algebra} is dedicated to definition of Bethe algebras. Section \ref{sec:duality} contains our main result, Theorem \ref{thm:duality between bethe algebras}, and its proof. 

\medskip

{\bf Acknowledgments.} 
We are grateful to K. Lu, V. Tarasov, F. Uvarov for interesting discussions and useful remarks. 
CH thanks B. Vicedo for motivation and fruitful conversations. EM thanks L. 
Banchi for a curious question and the follow up discussions which prompted this project.

This work was partially supported by a grant from the Simons Foundation \#353831.
 
\section{Preliminaries}\label{sec:preliminary}
\subsection{Superspaces and superalgebras}\label{subsec:superspace and superalgebra}
Let $\bs s = (s_1,\dots , s_{m+n} )$, $s_i\in\{\pm1\}$, be a sequence such that $1$ occurs exactly $m$ times. We call such a sequence a {\it parity sequence}. Denote by $S_{m|n}$ the set of all parity sequences. We call the parity sequence $\bs s_0 = (1,\dots, 1, -1,\dots,-1)$ {\it standard}. 

\medskip

We always work over $\C$.
A \emph{vector superspace} $V=V_{\bar{0}}\oplus V_{\bar{1}}$ is a $\Z_2$-graded vector space. The \emph{parity} of a homogeneous vector $v$ is denoted by $\bar{v}\in\Z/2\Z=\{\bar{0},\bar{1}\}$. An element $v$ in $V_{\bar{0}}$ (respectively, $V_{\bar{1}}$) is called \emph{even} (respectively, \emph{odd}), and we write $\bar{v}=\bar{0}$ (respectively, $\bar{v}=\bar{1}$). We set $(-1)^{\bar{0}}=1$ and $(-1)^{\bar{1}}=-1$. 

Let $\C^{m|n}$ be a complex vector superspace, with $\dim(\C^{m|n})_{\bar{0}}=m$ and $\dim(\C^{m|n})_{\bar{1}}=n$. Given a parity sequence $\bs s$, define 
\[
\bar{i}^{\bs s}=\begin{cases}
\bar{0},\qquad s_i=1,\\
\bar{1},\qquad s_i=-1,
\end{cases}
\] and choose a homogeneous basis $e^{\bs s}_i,\ i=1,\dots,m+n$, of $\C^{m|n}$ such that $\bar{e}^{\bs s}_i=\bar{i}^{\bs s}$. 

We often drop $\bs{s}_0$ from the notation depending on a parity sequence and write, for example, $\bar i=\bar{i}^{\bs {s}_0}$, $e_i=e_i^{\bs s_0}$, etc.  

A {\it superalgebra} is a vector superspace with an even, bilinear, associative, unital product operation. Given superalgebras $\mathcal{A},\mathcal{B}$,
the tensor product $\mathcal{A}\otimes\mathcal{B}$ is a superalgebra. For any homogeneous elements $x,x'\in\mathcal{A}$, $y,y'\in\mathcal{B}$, the product in the superalgebra $\mathcal{A}\otimes\mathcal{B}$ is 
\[
(x\otimes y)(x'\otimes y')=(-1)^{\bar{x}'\bar{y}}(xx'\otimes yy').
\]  For $x\in\mathcal{A}$, $a\in\{1,\dots,k\}$, denote $1^{\otimes(a-1)}\otimes x\otimes 1^{\otimes(k-a)}\in\mathcal{A}^{\otimes k}$ by $x^{(a)}$.

\subsection{The Lie superalgebra $\gl_{m|n}$}\label{subsec:lie superalgebra gl_{m|n}}
The \emph{Lie superalgebra $\gl_{m|n}$} is spanned  by $e_{i,j}$, $i,j=1,\dots,m+n$, with $\bar{e}_{i,j}=\bar{i}+\bar{j}$ (in the standard parity), and the {\it superbracket} is given by 
\be\label{equation:superbracket}
[e_{i,j},e_{p,q}]=\delta_{j,p}e_{i,q}-(-1)^{(\bar{i}+\bar{j})(\bar{p}+\bar{q})}\delta_{i,q}e_{p,j}.
\ee
The universal enveloping algebra of $\gl_{m|n}$ is denoted by $U\gl_{m|n}$. 

The \emph{Cartan subalgebra} $\mathfrak{h}$ of $\gl_{m|n}$ is spanned by $e_{i,i}$, $i=1,\dots,m+n$. The \emph{weight space} $\mathfrak{h}^*$ is the dual space of $\mathfrak{h}$. Let $\epsilon_i$, $i=1,\dots,m+n$, be a basis of $\mathfrak{h}^*$ such that $\epsilon_i(e_{j,j})=\delta_{i,j}$. We denote the weight $\epsilon_1+\dots +\epsilon_i$ by $\omega_i$, $i=0,\dots,m$. 

Given a weight $\la=\la_1\epsilon_1+\dots+\la_{m+n}\epsilon_{m+n}\in\mathfrak{h}^*$, we also denote $\la$ by the sequence $(\la_1,\dots,\la_{m+n})$. 

The \emph{nilpotent subalgebra} $\mathfrak{n}^+$ of ${\gl_{m|n}}$ (respectively $\mathfrak{n}^-$) is spanned by $e_{i,j}$, $i<j$ (respectively $e_{i,j}$, $i>j$).

\subsection{Representations of $\gl_{m|n}$}\label{subsec:representations} Let $V$ be a $\gl_{m|n}$-module.
Denote $\pi_V: U\gl_{m|n}\to \End(V)$ the corresponding map.

A non-zero vector $v\in V$ is called a \emph{vector of weight $\la$}, if $hv=\la(h)v$ for all $h\in\mathfrak{h}$. A non-zero vector $v$ of weight $\la$ is called a {\it singular vector}, if $\mathfrak{n}^+v=0$.  Denote the subspace of vectors of weight $\la$ in $V$ by $V[\la]$. Denote the subspace of singular vectors in $V$ by $V^{\rm sing}$. Denote $V[\la]\cap V^{\rm sing}$ by $V^{\textup{sing}}[{\la}]$. 

Denote the \emph{highest weight irreducible module of highest weight $\la$} by  $L(\la)$. 
The module  $L(\la)$ is generated by a singular vector $v_\la$ of weight $\la$. The vector $v_\la$ is called 
the \emph{highest weight vector of $L(\la)$}.

A module $V$ is called a \emph{polynomial module} if it is an irreducible submodule of $(\C^{m|n})^{\otimes N}$ for some $N\in \Z_{\geq 0}$. Polynomial modules are modules of the form $L(\mu^\natural)$ where the highest weights $\mu^\natural$ are described by $(m|n)$-hook partitions $\mu$ as follows.

Let $\mu=(\mu_1\geq\mu_2\geq\dots)$ be a partition: $\mu_i\in\Z_{\geq0}$ and $\mu_i=0$ for $i\gg0$. The partition is called an {\it $(m|n)$-hook partition} if $\mu_{m+1}\leq n$. For $i=1,\dots,\mu_1$, let $l_i(\mu)=\max\{j,\mu_j \geq i\}$ be the number of boxes of $\mu$ in the $i$-th column. We also set $l_{\mu_1+1}(\mu)=0$. We have $\mu_{l_i(\mu)}\geq i$ and $\mu_{l_i(\mu)+1}< i$, $i=1,\dots,\mu_1$. We call $l_1(\mu)$ the {\it length of $\mu$}.

Given an $(m|n)$-hook partition $\mu$, the corresponding weight $\mu^\natural=(\mu^\natural_1,\dots, \mu^\natural_{m+n})$ of the corresponding polynomial module $L(\mu^\natural)$ is given by
$\mu^\natural_i=\mu_i$, $i=1,\dots,m$, $\mu^\natural_{m+j}=\max\{l_j(\mu)-m,0\}$, $j=1,\dots,n$. 
We call weights $\mu^\natural$ corresponding to $(m|n)$-partitions the \emph{polynomial $\gl_{m|n}$ weights}.

\subsection{The $\gl_{m|n}$ current algebra and the evaluation modules}\label{gl_{m|n} current algebra and its evaluation modules}
Let $t$ be an even variable.
Let $\gl_{m|n}[t]=\gl_{m|n}\otimes \C[t]$ be the Lie superalgebra of $\gl_{m|n}$ valued polynomials with pointwise superbracket. We call $\gl_{m|n}[t]$ the {\it current algebra}. Denote by $U\gl_{m|n}[t]$ the universal enveloping algebra of $\gl_{m|n}[t]$.

We identify the Lie superalgebra $\gl_{m|n}$ with the subalgebra $\gl_{m|n}\otimes 1$ of constant polynomials in $\gl_{m|n}[t]$. Therefore any $\gl_{m|n}[t]$-module has the canonical structure of a $\gl_{m|n}$-module.

The standard generators of $\gl_{m|n}[t]$ are $e_{i,j}\otimes t^r$, $i,j=1,\dots,m+n$, $r\in\Z_{\geq 0}$. The superbracket is given by  
\be\label{equation:current algebra commutator}
(u-v)[e_{i,j}(u),e_{p,q}(v)]=-[e_{i,j},e_{p,q}](u)+[e_{i,j},e_{p,q}](v),
\ee where
\be\label{equation:formal power series}
e_{i,j}(v)=\sum_{r=0}^{\infty} (e_{i,j}\otimes t^r)v^{-r-1}
\ee are the formal power series.

For each $z\in\C$, there exists a {\it shift of spectral parameter automorphism} $\rho_z$ of $\gl_{m|n}[t]$ sending $g(v)$ to $g(v-z)$ for all $g\in \gl_{m|n}$. Given a $\gl_{m|n}[t]$-module $V$, denote by $V_z$ the pull-back of $V$ through the automorphism $\rho_z$. As $\gl_{m|n}$-modules, $V$ and $V_z$ are isomorphic by the identify map.

We have the {\it evaluation homomorphism}, $ev: \gl_{m|n}[t]\rightarrow\gl_{m|n}$, $ev:g(v)\mapsto gv^{-1}$. For any $\gl_{m|n}$-module $V$, denote by the same letter the $\gl_{m|n}[t]$-module, obtained by pull-back of $V$ through the evaluation homomorphism $ev$. Given a $\gl_{m|n}$-module $V$ and $z\in\C$, the $\gl_{m|n}[t]$-module $V_z$ is called an {\it evaluation module}. The action of $\gl_{m|n}[t]$ in $V_z$ is given by
\be\label{equation:current algebra action}
e_{i,j}(v)w=\frac{e_{i,j}w}{v-z},
\ee for any $w\in V$, $i,j=1,\dots,m+n$.

Note that if $\la^{(1)},\dots,\la^{(k)}$ are polynomial weights and $z_1,\dots,z_k$ are pairwise distinct complex numbers, then the module $\otimes_{a=1}^{k}L(\la^{(a)})_{z_a}$ is irreducible. 

\section{Berezinians of affine Manin matrices}\label{sec:Ber}

In this section, we recall some facts about Berezinians, following \cite{MR}. We give a definition of Berezinians of affine Manin matrices to arbitrary parities and study its properties. 

Let $\mathcal{A}$ be a superalgebra. Given a matrix $A=\big(a_{i,j}\big)_{i,j=1,\dots,m+n}$, $a_{i,j}\in\mathcal{A}$, with a two sided inverse $A^{-1}$, we denote the $(i,j)$ entry of $A^{-1}$ by $\wt{a}_{i,j}$.

\subsection{Berezinian of standard parity}\label{subsec: Ber}
Let $A=\big(a_{i,j}\big)_{i,j=1,\dots,m+n}$ be a matrix  with a two sided inverse. The {\it Berezinian of standard parity} of $A$, see \cite{MR}, is
\begin{align}\label{equation:ber}
\ber A=\Big(\sum_{\si\in\mathfrak{S}_m}\sgn\si\  a_{\si(1),1}\dots a_{\si(m),m}\Big)\times\Big(\sum_{\tau\in\mathfrak{S}_n}\sgn\tau\ \wt{a}_{m+1,m+\tau(1)}\dots \wt{a}_{m+n,m+\tau(n)}\Big),
\end{align} where $\sS_r$ is the symmetric group on $r$ letters. In the case of $n=0$, the above formula is the {\it column determinant} which we denote by $\cdet A$. In the case of $m=0$, the above formula is the {\it row determinant} of the inverse matrix which we denote by ${\rm {rdet}}\ A^{-1}$.

We call $A=\big(a_{i,j}\big)_{i,j=1,\dots,m+n}$, $a_{i,j}\in\mathcal{A}$, a {\it matrix of standard parity over $\mathcal{A}$}, if $\bar{a}_{i,j}=\bar{i}+\bar{j}$. 

We call $A$ a {\it Manin matrix of standard parity}, if $A$ is of standard parity and
\[
[a_{i,j},a_{p,q}]=(-1)^{\bar{i}\bar{j}+\bar{i}\bar{p}+\bar{j}\bar{p}}[a_{p,j},a_{i,q}],\qquad i,j,p,q=1,\dots,m+n.
\]
Many properties of even Manin matrices are known, see \cite{CFR}. Similar properties can be proved in the supersymmetric case, but we need here only a couple of facts which we extract from \cite{MR}.

Let $w$ be an even formal variable. We call $A(w)=\big(a_{i,j}(w)\big)_{i,j=1,\dots,m+n}$ an {\it affine matrix}, if 
\[
a_{i,j}(w)=\sum_{r=0}^{\infty}a_{i,j,r}w^{r},\;a_{i,j,r}\in\mathcal{A},\qquad a_{i,j,0}=\de_{i,j},\ i,j=1,\dots,m+n.
\] 

In other words, an affine matrix is a matrix whose entries $a_{i,j}(w)\in\mathcal{A}[[w]]$ are formal power series in variable $w$ and such that $A(0)=I$. In particular, every affine matrix has a two sided inverse.

Given a Manin matrix $A$ of standard parity, the matrix $(1+wA)$ is an affine Manin matrix of standard parity.

\begin{lem}\cite{MR}\label{lem:inverse is manin}
	Let $A(w)$ be an affine Manin matrix of standard parity. Then the inverse matrix $A^{-1}(w)$ is an affine Manin matrix of standard parity. \qed
\end{lem}

For an arbitrary $(m+n)\times(m+n)$ matrix $A$ with a two sided inverse, the {\it $(i,j)$ quasideterminant of $A$} is $\wt{a}_{j,i}^{-1}$. If $\wt{a}_{j,i}^{-1}$ does not exist in $\mathcal A$, then the $(i,j)$ quasideterminant of $A$ is not defined. We write 
\[
\wt{a}_{j,i}^{-1}=\bpm
& a_{1,1} 	& \dots & a_{1,j} &\dots & a_{1,m+n}\\
& \dots	& \dots & \dots  &\dots & \dots	\\
& a_{i,1}& \dots & \fbox{$a_{i,j}$} &\dots &a_{i,m+n}\\
& \dots	& \dots & \dots  &\dots & \dots	\\
& a_{m+n,1}&\dots &a_{m+n,j} &\dots &a_{m+n,m+n}
\epm.
\] 
For $i=1,\dots,m+n$, define the {\it principal quasi-minors of $A$} by
\be\label{equation:d_i}
d_i(A)=\bpm
	& a_{1,1} 	& \dots & a_{1,i}\\
	& \dots	& \dots & \dots \\
	& a_{i,1}& \dots & \fbox{$a_{i,i}$}
\epm. 
\ee 
\medskip

If $A(w)$ is an affine matrix, then the principal quasi-minors $d_i(A(w))$, $i=1,\dots,m+n$, are well defined.

The Berezinian of Manin matrices of standard parity is computed in terms of quasi-minors.
\begin{thm}\cite{MR}\label{thm:MR quasideterminant}
Let $A(w)$ be an affine Manin matrix of standard parity. The Berezinian $\ber A(w)$ admits the quasideterminant factorization:
\begin{align*}\label{equation: MR quasideterminant}
\ber A(w)=d_1(A(w))\dots &d_{m}(A(w))\times d_{m+1}^{-1}(A(w))\dots d_{m+n}^{-1}(A(w)).
\end{align*}
\qed
\end{thm}

\subsection{Berezinian of general parity}\label{subsec:Berezinian of different parities}
Fix a parity sequence $\bs s\in S_{m|n}$, see Section \ref{subsec:superspace and superalgebra}.

We call $A=\big(a_{i,j}\big)_{i,j=1,\dots,m+n}$, $a_{i,j}\in\mathcal{A}$,  a {\it matrix of parity $\bs s$}, if $\bar{a}_{i,j}=\bar{i}^{\bs s}+\bar{j}^{\bs s}$.  
Note that $0$ is both odd and even, in particular, the zero and the identity matrices are matrices of arbitrary parity $\bs s$.

We call $A$ a {\it Manin matrix of parity $\bs s $} if $A$ is of parity $\bs s$ and 
\[
[a_{i,j},a_{p,q}]=(-1)^{\bar{i}^{\bs s}\bar{j}^{\bs s}+\bar{i}^{\bs s}\bar{p}^{\bs s}+\bar{j}^{\bs s}\bar{p}^{\bs s}}[a_{p,j},a_{i,q}],\qquad i,j,p,q=1,\dots,m+n.
\] 

The symmetric groups $\sS_{m+n}$ acts on matrices and parities by the following rule.
For $\si\in\sS_{m+n}$, we set  $\si(A)=\si A\si^{-1}=\big(a_{\si^{-1}(i),\si^{-1}(j)}\big)_{i,j=1,\dots,m+n}$ and $\si(\bs s)=(s_{\si^{-1}(1)},\dots,s_{\si^{-1}(m+n)})$.

The following lemma is straightforward. 
\begin{lem}\label{lem:permutation of manin is manin}
	Let $A$ be a Manin matrix of parity $\bs s $. Then $\si(A)$ is a Manin matrix of parity $\si(\bs s)$.\qed
\end{lem}

Lemma \ref{lem:inverse is manin} is extended to affine Manin matrices of arbitrary parities.
\begin{lem}\label{lem:inverse is manin of parity}
	Let $A(w)$ be an affine Manin matrix of parity $\bs s$. Then $A^{-1}(w)$ is an affine Manin matrix of parity ${\bs s}$.
\end{lem}
\begin{proof}
There exists $\si\in\sS_{m+n}$ such that $\si(\bs s)=\bs s_0$. By Lemma \ref{lem:permutation of manin is manin}, $\si(A(w))$ is an affine matrix of standard parity. By Lemma \ref{lem:inverse is manin}, the matrix 
$(\si(A(w)))^{-1}$ is an affine Manin matrix of standard parity. We have $(\si(A(w)))^{-1}=\si(A^{-1}(w))$. Therefore by Lemma \ref{lem:permutation of manin is manin}, the matrix $A^{-1}(w)=\si^{-1}((\si(A(w)))^{-1})$ is an affine Manin matrix of parity $\bs s$.
\end{proof}

\medskip

Let $A(w)$ be an affine Manin matrix of parity $\bs s$. We define the {\it Berezinian of parity $\bs s$} of $A(w)$ by 
\be\label{equation:ber for any parity}
\bers A(w)=d_1^{s_1}(A(w))\dots d_{m+n}^{s_{m+n}}(A(w)).
\ee By Theorem \ref{thm:MR quasideterminant}, definition \eqref{equation:ber for any parity} coincides with definition \eqref{equation:ber} in the case of standard parity.

\medskip

Let $A(w)$ be an affine Manin matrix of parity $\bs s$. 
Fix $r\in\{1,\dots,m+n\}$ and consider the corresponding blocks. Namely, let $W(w),X(w),Y(w),Z(w)$ be submatrices of $A(w)=\bpm W(w) & X(w)\\ Y(w)&Z(w)\epm$ 
of size $r\times r$, $r\times(m+n-r)$, $(m+n-r)\times r$, and $(m+n-r)\times (m+n-r)$ respectively.

 Then $W(w)$ and $Z(w)$ are affine Manin matrices of parities $\bs{s}|^r$ and $\bs{s}|_{m+n-r}$, where 
$\bs{s}|^r=(s_1,\dots,s_r)$ and $\bs{s}|_{m+n-r}=(s_{r+1},\dots,s_{m+n})$.

We have the Gauss decomposition:
\be\label{inv}
			A(w)=\bpm W(w) & X(w)\\ Y(w)&Z(w)\epm=\bpm1&0\\Y(w)W^{-1}(w)&1\epm\bpm W(w) & X(w)\\ 0&Z(w)-Y(w)W^{-1}(w)X(w)\epm.
	\ee
The next proposition claims that the Gauss decomposition is compatible with the definition of Berezinian.
\begin{prop}\label{prop:blockwise change}
The matrices $W(w)$ and $Z(w)-Y(w)W^{-1}(w)X(w)$ are affine Manin matrices.  We have
	\be\label{equation:blockwise ber}
	\bers A(w)={\rm Ber}^{\bs{s}|^r}\, W(w)\times {\rm Ber}^{\bs{s}|_{m+n-r}}\, \left(Z(w)-Y(w)W^{-1}(w)X(w)\right).
	\ee
\end{prop}
\begin{proof}
The matrix $\big(Z(w)-Y(w)W^{-1}(w)X(w)\big)^{-1}$ is a submatrix of $A^{-1}(w)$, see \eqref{inv}. Therefore, by Lemma \ref{lem:inverse is manin of parity}, the matrix $\big(Z(w)-Y(w)W^{-1}(w)X(w)\big)^{-1}$ is an affine Manin matrix of parity $\bs s|_{m+n-r}$, which implies in turn that $Z(w)-Y(w)W^{-1}(w)X(w)$ is an affine Manin matrix of parity $\bs s|_{m+n-r}$.

	For $i=r+1,\dots,m+n$, denote by $X(w)|_{i}$ the submatrix of size $r\times(i-r)$  formed by the first $(i-r)$ columns of $X(w)$, denote by $Y(w)|^{i}$ the submatrix of size $(i-r)\times r$ formed by the first $(i-r)$ rows of $Y(w)$, and denote by $Z(w)|_{i}^{i}$ the top left $(i-r)\times(i-r)$ submatrix of $Z(w)$. Similar to \eqref{inv}, we have
	\be\label{equation:block decompo}
			\bpm W(w) & X(w)|_{i}\\ Y(w)|^{i}&Z(w)|_{i}^{i}\epm^{-1}\hspace{-10pt}=\bpm W(w) & X(w)|_{i}\\ 0&Z(w)|_{i}^{i}-Y(w)|^{i}W^{-1}(w)X(w)|_{i}\epm^{-1}\hspace{-5pt}\bpm1&0\\-Y(w)|^{i}W^{-1}(w)&1\epm.
	\ee
	
	From the definition of principal quasi-minors, we have $d_i(A(w))=d_i(W(w))$, $i=1,\dots,r$.
	From \eqref{equation:block decompo}, we have 
	\be\label{equation: principal quasi-minors}
	d_i(A(w))=d_{i-r}(Z(w)-Y(w)W^{-1}(w)X(w)),\ i=r+1,\dots,m+n.
	\ee
\end{proof}
Now we can prove that the action of $\sS_{m+n}$ does not change the Berezinian.
\begin{prop}\label{prop:manin matrix row-column change} Let $A(w)$ be an affine Manin matrix of parity $\bs s$. Let $\si\in\sS_{m+n}$.
	We have \be\label{equation:manin matrix row-column change}
	\bers A(w)={\rm Ber}^{\si(\bs s)}\, \si(A(w)).
	\ee
\end{prop}
\begin{proof}
	It suffices to consider $\si=(i,i+1)$, $i=1,\dots,m+n-1$. Moreover, it is sufficient to show \[d_i^{s_i}(A(w))d_{i+1}^{s_{i+1}}(A(w))=d_i^{s_{i+1}}(\si(A(w)))d_{i+1}^{s_i}(\si(A(w))).\]
	Without losing generality we treat the case $i=m+n-1$.
	
	Consider the block decomposition of $A(w)$ with $r=m+n-2$. In particular, $Z(w)$ is a $2\times 2$ matrix. 
	By \eqref{equation: principal quasi-minors} with $i=m+n-1, m+n$, \[
	d_{m+n-1}(A(w))=d_1(Z(w)-Y(w)W^{-1}(w)X(w)),\quad d_{m+n}(A(w))=d_2(Z(w)-Y(w)W^{-1}(w)X(w)),
	\]
	and
	\begin{align*}
		&d_{m+n-1}(\si (A(w)))=d_1(\bar\si(Z(w)-Y(w)W^{-1}(w)X(w))),\\ &d_{m+n}(\si(A(w)))=d_2(\bar\si(Z(w)-Y(w)W^{-1}(w)X(w))),
	\end{align*} 
	where $\bar \si=(1,2)\in\sS_2$.
	
	Thus, the proposition is reduced to the case of $2\times 2$ affine Manin matrices. This is proved by a direct computation.
\end{proof}

\subsection{Affine-like Manin matrices}
We extend the results on Berezinians of affine matrices to another class of matrices which we call affine-like matrices.

Denote $\mathcal{A}((w))$  the superalgebra of formal Laurent series in $w$ with coefficients in $\mathcal{A}$,
$$
\mathcal{A}((w))=\{\sum_{r=-N}^\infty b_r w^r, \ N\in\Z,\ b_r\in\mathcal A\}.
$$

Let $A=\big(a_{i,j}\big)_{i,j=1,\dots,m+n}$ be a matrix of parity $\bs s$ with entries $a_{i,j}$ in $\mathcal A$.
We call $A$ an {\it affine-like matrix of parity $\bs s$} if the following two conditions are met:
\begin{itemize}
\item for any subset $\mathfrak a\subset\{1,\dots,m+n\}$, the matrix $A_{\mathfrak a}=\big(a_{i,j}\big)_{i,j\in\mathfrak a}$ has a two sided inverse with entries in $\mathcal A$ and the diagonal entries of $A^{-1}_{\mathfrak a}$ are invertible in $\mathcal A$. 
\item there exists an injective homomorphism of superalgebras $\Phi_A:\mathcal A\to \mathcal A((w))$ such that $a_{i,j}\mapsto a_{i,j}+\de_{i,j}w^{-1}$. 
\end{itemize}

If $A$ is an affine-like matrix, then the principal quasi-minors $d_i(A)$ are well-defined. If $A$ is an affine-like matrix then $\si(A)$ is affine-like for any $\si\in\sS_{m+n}$.

Our definition is motivated by the following simple observation.

\begin{lem}
If $A$ is an affine-like matrix, then $w\Phi_A(A)=1+wA$ is an affine matrix. Moreover, we have $\Phi_A(A^{-1})=(\Phi_A(A))^{-1}$ and
$\Phi_A(d_i(A))=d_i(\Phi_A(A))$, $i=1,\dots,m+n$. 

If  $A$ is an affine-like Manin matrix of parity $\bs s$, then $w\Phi_A(A)$ is an affine Manin matrix  of parity $\bs s$ and $A^{-1}$ is also an affine-like Manin matrix of parity $s$. \qed
\end{lem}

Now we can extend the definition of the Berezinian and its properties to affine-like matrices.

Let $A$ be an affine-like Manin matrix of parity $\bs s$. Define Berezinian $\bers A$ by formula \eqref{equation:ber for any parity}.

\begin{prop}\label{prop:affine-like properties}
Propositions \ref{prop:blockwise change} and \ref{prop:manin matrix row-column change} hold for affine-like Manin matrices of parity $\bs s$.\qed
\end{prop}

\section{Bethe algebra $\BB_{m|n}(\bs\La)$}\label{sec:gl_{m|n} bethe algebra}
In this section we discuss Bethe subalgebras $\BB_{m|n}(\bs\La)\subset U\gl_{m|n}[t]$. The Bethe subalgebras  $\BB_{m|n}(\bs\La)$ are commutative and depend on parameters $\bs\La=(\La_1,\dots,\La_{m+n})\in\C^{m+n}$. 

\subsection{Algebra of pseudodifferential operators}
Let $\mathcal{A}$ be a differential superalgebra with an even derivation $\dd:\mathcal A \to \mathcal A$.  For $r\in\Z_{\geq 0}$, denote  the $r$-th derivative of $a\in\mathcal{A}$ by $a_{(r)}$.

Let $\mathcal{A}((\dd^{-1}))$ be the {\it algebra of pseudodifferential operators}. The elements of $\mathcal{A}((\dd^{-1}))$ are Laurent series in $\dd^{-1}$ with coefficients in $\mathcal A$, and the product follows from the relations
\[
\dd\dd^{-1}=\dd^{-1}\dd=1,\quad\dd^{r}a=\sum_{s=0}^{\infty}\binom{r}{s}a_{(s)}\dd^{r-s},\ r\in\Z,\  a\in\mathcal{A},
\] where \[\binom{r}{s}=\frac{r(r-1)\dots (r-s+1)}{s!}.\]

Let $\mathcal{A}[\dd]\subset \mathcal{A}((\dd^{-1})) $ be the {\it subalgebra of differential operators}, 
\[\mathcal{A}[\dd]=\{\sum_{r=0}^{M}a_r\dd^{r},M\in\Z_{\geq0},a_r\in\mathcal{A}\}.\]

Consider a linear map $\Phi: \mathcal{A}((\dd^{-1}))\rightarrow\mathcal{A}[\dd]((w)),$
\be\label{equation:Phi map}
\Phi:\sum_{r=-\infty}^{N}a_r\dd^{r} \mapsto\sum_{r=-\infty}^{N}a_r(w^{-1}+\dd)^{r},
\ee 
where the right hand side is expanded by the rule $(w^{-1}+\dd)^{r}=\sum_{s=0}^{\infty}\binom{r}{s}\dd^{s}w^{-r+s}$.

\begin{lem}\label{lem:pseudo to differ}
	The map $\Phi$ is an injective homomorphism of superalgebras.
\end{lem}
\begin{proof}
	For any $r$, the coefficient of $w^{r}$ in the right hand side of \eqref{equation:Phi map} is a summation of finitely many terms.
	
	The coefficient of $w^{-N}$ in $\Phi(\sum_{r=-\infty}^{N}a_r\dd^{r})$ is $a_N$. Therefore, $\Phi$ is injective.
	
	For any $a\in\mathcal{A}$, we have 
	\[
	\Phi(\dd^{r}a)=\Phi\left(\sum_{s=0}^{\infty}\binom{r}{s}a_{(s)}\dd^{r-s}\right)=\sum_{s=0}^{\infty}\sum_{t=0}^{\infty}\binom{r}{s}\binom{r-s}{t}a_{(s)}\dd^{t}w^{-r+s+t}.
	\] 
	Then, changing the summation indices we obtain
	\[
	\Phi(\dd^r)\Phi(a)=\Phi(\dd^{r})a=\sum_{s=0}^{\infty}\binom{r}{s}\dd^{s}w^{-r+s}a=\sum_{s=0}^{\infty}\sum_{t=0}^{s}\binom{r}{s}\binom{s}{t}a_{(t)}\dd^{s-t}u^{-r+s}=\Phi(\dd^ra).
	\] 
	Therefore, the map $\Phi$ is a homomorphism of superalgebras.
\end{proof}  

\subsection{Bethe subalgebra}\label{subsec:balgebra} Let
$$ 
\mathcal A_v^{m|n} = U\gl_{m|n}[t]((v^{-1}))=\Big\{\sum_{r=-\infty}^N g_r v^r, \ N\in\Z, \ g_r\in U\gl_{m|n}[t]\  \Big\}
$$
be the superalgebra of Laurent series in $v^{-1}$ with coefficients in $U\gl_{m|n}[t]$.
The algebra $\mathcal A_v^{m|n}$ is a differential superalgebra with derivation $\dv$.

Let $\bs\La=(\La_1,\dots,\La_{m+n})$ be a sequence of complex numbers. Consider the matrix $B(\bs\La)$ with entries in the algebra of pseudodifferential operators $\mathcal A_v^{m|n} ((\dv^{-1}))$ given by
\be\label{equation: universal oper}
B(\bs\La)=\Big(\de_{i,j}(\dv-\La_i)-(-1)^{\bar i}e_{i,j}(v)\Big)_{i,j=1,\dots,m+n}.
\ee  
The following lemma is checked by a straightforward computation.
\begin{lem}
The matrix $B(\bs\La)$ is an affine-like Manin matrix of standard parity with the map $\Phi_{B(\bs\La)}=\Phi$, see \eqref{equation:Phi map}. \qed
\end{lem}

Consider the expansion of the Berezinian of the affine Manin matrix $w\Phi(B(\bs\La))=1+wB(\bs\La)$: 
\be\label{equation:Molev expansion}
\ber(1+wB(\bs\La))=\sum_{r=0}^{\infty}\sum_{s=0}^{r}B_{r,s}^{\bs \La}(v)\dv^{r-s}w^r,
\ee
where $B_{r,s}^{\bs \La}(v)\in \mathcal A_v^{m|n}$.
The following fundamental result is known.
\begin{thm}\cite{MR}\label{thm:bethe algebra}
The series $B_{r,s}^{\bs \La}(v)$ pairwise commute, $[B_{r_1,s_1}^{\bs \La}(v_1),B_{r_2,s_2}^{\bs \La}(v_2)]=0$, for all $r_1,s_1,r_2,s_2$.

The series $B_{r,s}^{\bs \La}(v)$ commute with the Cartan subalgebra ${\mathfrak h}\subset U\gl_{m|n}$,  
$[B_{r,s}^{\bs \La}(v),e_{i,i}]=0$, for all $r,s,i$.
\qed
\end{thm} 
We call the commutative subalgebra generated by coefficients of series $B_{r,s}^{\bs \La}(v)$, $r,s\in\Z_{\geq 0}$, $s\leq r$, the {\it Bethe subalgebra of $U\gl_{m|n}[t]$} and denote it by $\BB_{m|n}(\bs\La)$.

\medskip

Alternatively, we can expand $\ber B(\bs \La)$ directly
\be\label{equation:pseudo expansion}
\ber B(\bs\La)=\sum_{r=-\infty}^{m-n}B_{r}^{\bs \La}(v)\dv^{r},
\ee 
where $B_{r}^{\bs \La}(v)\in \mathcal A_v^{m|n}$.
\begin{prop}\label{cor:Two Ber coincides as Bethe algebra}
	The coefficients of the series $B^{\bs\La}_{r}(v)$, $r\in\Z_{\leq m-n}$ generate the Bethe algebra $\BB_{m|n}(\bs\La)$. 
\end{prop}
\begin{proof}
	We have 
	$$
	w^{m-n}\Phi(\ber B(\bs\La))=w^{m-n}\ber \Phi (B(\bs\La))=\ber(1+wB(\bs\La)),
	$$
	since $\Phi$ is a homomorphism of superalgebras by Lemma \ref{lem:pseudo to differ}. Moreover, $\Phi(a)=a$, for $a\in \mathcal A_v^{m|n}$. The proposition follows.
\end{proof}

\section{Duality between $\BB_{m|n}$ and $\BB_{k}$}\label{sec:duality}
In this section we show the duality between $\BB_{m|n}(\bs\La)$ and $\BB_{k}(\bs z)$ acting in the space of supersymmetric polynomials. The duality in the case of $n=0$ is given in \cite{MTV5}.

\subsection{The duality between $\gl_{m|n}$ and $\gl_{k}$}\label{classical duality}
We start with the standard duality between $\gl_{m|n}$ and $\gl_{k}$.

Let $\mathcal{D}$ be the superalgebra generated by $x_{i,a},\dd_{i,a},i=1,\dots,m+n,a=1,\dots,k,$ with parity given by $\bar{x}_{i,a}=\bar{\dd}_{i,a}=\bar{i}$ and the relations given by supercommutators 
$$
[x_{i,a},x_{j,b}]=[\dd_{i,a},\dd_{j,b}]=0,\qquad [\dd_{i,a},x_{j,b}]=\de_{i,j}\de_{a,b},\ \mbox{for all}\ i,j,a,b.
$$ 

Let $V\subset \mathcal D$ be the subalgebra generated by $x_{i,a}$, $i=1,\dots,m+n$, $a=1,\dots,k$. Then 
$$
V=\C[x_{i,a}, \ i=1,\dots,m,\ a=1,\dots,k ]\otimes\La(x_{j,a},\ j=m+1,\dots,m+n, \ a=1,\dots, k)
$$ 
is the product of a polynomial algebra and a Grassmann algebra. We call $V$ the {\it space of supersymmetric polynomials} or {\it bosonic-fermionic space}. The algebra $\mathcal D$ acts on $V$ in the obvious way.

We have a homomorphism of superalgebras $\pi_{m|n}:\gl_{m|n} \to \mathcal D$ given by
$$
\pi_{m|n} (e^{[m|n]}_{i,j})=\sum_{a=1}^{k}x_{i,a}\dd_{j,a},\ i,j=1,\dots,m+n,
$$
where we write the suffix  in $e^{[m|n]}_{i,j}$ to indicate that these are elements of $\gl_{m|n}$. 
In particular, $\gl_{m|n}$ acts on $V$.

For $a\in\{1,\dots,k\}$, let $V^{(a)}_{m|n}\subset V$ be the subalgebra generated by $x_{1,a},\dots,x_{m+n,a}$. Then we have isomorphisms of $\gl_{m|n}$-modules:
\[
V^{(a)}_{m|n}=\bigoplus_{d=0}^{\infty} L^{(a)}_{m|n}(d\epsilon_1), \qquad V=\bigotimes_{a=1}^k V^{(a)}_{m|n},
\] 
where $L^{(a)}_{m|n}(d\epsilon_1)$ is the the irreducible $\gl_{m|n}$-module with highest weight $(d,0,\dots,0)$ and highest weight vector $x_{1,a}^d$. The submodule $L^{(a)}_{m|n}(d\epsilon_1)$
is spanned by all monomials of total degree $d$ in $V^{(a)}_{m|n}$.

\medskip

We also have the homomorphism of superalgebras $\pi_k:\gl_k \to \mathcal D$ given by
\[
\pi_k(e^{[k]}_{a,b})=\sum_{i=1}^{m+n}x_{i,a}\dd_{i,b},\qquad a,b=1,\dots,k. 
\]
In particular, $\gl_k$ also acts on $V$.

For $i\in\{1,\dots,m+n\}$, let $V^{(i)}_k\subset V$ be the subalgebra generated by $x_{i,1},\dots,x_{i,k}$. If $i\leq m$, the space $V^{(i)}_k$ is the polynomial ring of $k$ variables, otherwise the space $V^{(i)}_k$ is the Grassmann algebra of $k$ variables. Then we have isomorphisms of $\gl_k$-modules:
\begin{align*}
V^{(i)}_{k}=\bigoplus_{d=0}^{\infty} L^{(i)}_k(d\epsilon_1), \  i\leq m,\qquad
V^{(i)}_{k}=\bigoplus_{a=0}^{k} L^{(i)}_k(\omega_a), \  i> m, \qquad
 V=\bigotimes_{i=1}^{m+n} V^{(i)}_k.
\end{align*}
Here, $L^{(i)}_k(d\epsilon_1)$, $i\leq m$, is the irreducible $\gl_k$-module with highest weight $(d,0,\dots,0)$ and highest weight vector $x_{i,1}^d$. 
The submodule $L^{(i)}_k(d\epsilon_1)$ is spanned by all monomials of total degree $d$ in $V^{(i)}_k$.
The module $L^{(i)}_k(\omega_a)$, $i>m$, is the irreducible $\gl_k$-module with highest weight $(\underbrace{1,\dots,1}_a,0\dots,0)$ and highest weight vector $x_{i,1}\dots x_{i,a}$. This submodule  is spanned by all monomials of total degree $a$ in $V^{(i)}_k$. 

In particular we have the canonical identification of weight spaces:
\begin{align}
&\left(L_{m|n}^{(1)}(\la_1\epsilon_1)\otimes \dots \otimes L_{m|n}^{(k)}(\la_k\epsilon_1)\right)[(\mu_1,\dots,\mu_{m+n})] \notag\\
&=\left(L_k^{(1)}(\mu_1\epsilon_1)\otimes\dots\otimes L_k^{(m)}(\mu_m\epsilon_1)\otimes L_k^{(m+1)}(\omega_{\mu_{m+1}})\otimes \dots\otimes L_k^{(m+n)}(\omega_{\mu_{m+n}})\right)[(\la_1,\dots,\la_{k})].\label{weight duality}
\end{align}
These weight spaces are spanned by monomials in $V$ which have total degree $\la_a$ with respect to variables $x_{1,a},\dots,x_{m+n,a}$ and total degree $\mu_i$ with respect to variables $x_{i,1},\dots, x_{i,k}$.

The standard duality between $\gl_{m|n}$ and $\gl_{k}$ is the following well-known statement.
 
\begin{lem}
The actions of $\gl_{m|n}$ and $\gl_k$ on $V$ commute. We have the isomorphism of $\gl_{m|n}\oplus\gl_k$
modules
\[
V=\bigoplus_{\mu\in P_{m,n;k}}L_{m|n}(\mu^\natural)\otimes L_k(\mu),
\]
where $P_{m,n;k}$ is the set of all $(m|n)$-hook partition with length at most $k$.\qed
\end{lem}

\subsection{The duality of Bethe algebras $\BB_{m|n}(\bs \La)$ and $\BB_k(\bs z)$}
Let $\bs z=(z_1,\dots,z_k)$ and $\bs\La=(\La_1,\dots,\La_{m+n})$ be two sequences of complex numbers. We extend actions of $\gl_{m|n}$ and $\gl_k$ on $V$ to the actions of the current algebras $\gl_{m|n}[t]$ and $\gl_{k}[t]$ as follows.

Let $\hat\pi_{m|n}: U\gl_{m|n}[t] \to \mathcal D$ and $\hat\pi_k: U\gl_k[t] \to \mathcal D$ be homomorphisms of superalgebras given by 
\begin{align}
\hat\pi_{m|n}:e^{[m|n]}_{i,j}(v)&\mapsto \sum_{a=1}^{k}\frac{x_{i,a}\dd_{j,a}}{v-z_a},\qquad i,j=1,\dots,m+n,\label{equation:lie superalgebra action}\\
\hat\pi_{k}:e^{[k]}_{a,b}(u)&\mapsto \sum_{i=1}^{m+n}\frac{x_{i,a}\dd_{i,b}}{u-\La_i},\qquad a,b=1,\dots,k.\label{equation:lie algebra action}
\end{align} 
Then the $\gl_{m|n}$-module $V^{(a)}_{m|n}$ becomes evaluation $\gl_{m|n}[t]$-module  $(V^{(a)}_{m|n})_{z_a}$ and the $\gl_{k}$-module $V^{(i)}_{k}$ becomes evaluation $\gl_{k}[t]$-module  $(V^{(i)}_{k})_{\La_i}$, see \eqref{equation:current algebra action}.

\medskip

The actions of $\gl_{m|n}[t]$ and $\gl_{k}[t]$ on $V$ do not commute anymore. However, we prove the theorem saying that the actions of Bethe algebras $\BB_{m|n}(\bs\La)\subset U\gl_{m|n}[t]$ and $\BB_{k}(\bs z)\subset U\gl_{k}[t]$ on $V$ coincide.

Recall that the Bethe algebra $\BB_{m|n}(\bs\La)$ is generated by the coefficients of the Berezinian of the matrix
\[
B(\bs\La)=\Big(\de_{i,j}(\dv-\La_i)-(-1)^{\bar i}e^{[m|n]}_{i,j}(v)\Big)_{i,j=1,\dots,m+n}.
\] 
Similarly, the Bethe algebra $\BB_{k}(\bs z)$ is generated by the coefficients of the column determinant of the matrix
\[
G(\bs z)=\Big(\de_{a,b}(\du-z_a)-e^{[k]}_{a,b}(u)\Big)_{a,b=1,\dots,k}.
\] 

\begin{thm}\label{thm:duality between bethe algebras}
The Bethe algebras $\hat\pi_{k}\BB_{k}(\bs z)$ and $\hat\pi_{m|n}\BB_{m|n}(\bs\La)$ coincide. 

Moreover, we have the following identification of generators. If $b_{r,s}(\bs z,\bs\La),\ g_{r,s}(\bs\La,\bs z)\in \mathcal D$ do not depend on $v$, $\dd_v$, $u$, $\dd_u$, and
	\begin{align*}
	(v-z_1)\dots(v-z_k)\ \hat\pi_{m|n}\,\ber B(\bs\La)=\sum_{r=0}^{k}\sum_{s=-\infty}^{m-n}b_{r,s}(\bs z,\bs\La)v^r\dd_v^{s}\ ,\\
	\frac{(u-\La_1)\dots(u-\La_m)}{(u-\La_{m+1})\dots(u-\La_{m+n})}\ \hat\pi_{k}\,\cdet G(\bs z)=\sum_{r=-\infty}^{m-n}\sum_{s=0}^kg_{r,s}(\bs\La,\bs z)u^{r}\dd_u^s\ ,
	\end{align*}
	 then
		\be\label{equation:duality of bethe algebras}
	b_{r,s}(\bs z,\bs\La)=g_{s,r}(\bs\La, \bs z).\ee
\end{thm}
\begin{proof}
The proof of this theorem is given in Sections \ref{subsec:capelli formula} and \ref{subsec:another capelli}.
\end{proof}
By Theorem \ref{thm:bethe algebra}, Bethe algebras preserve weight spaces.
In particular, Theorem \ref{thm:duality between bethe algebras} gives an identification of action of Bethe algebras $\BB_{k}(\bs z)$ and $\BB_{m|n}(\bs\La)$ on the weight spaces \eqref{weight duality}.

\subsection{An identity of Capelli type}\label{subsec:capelli formula} 
In this section we give an explicit expansion of $\hat\pi_{k}\cdet G(\bs z)$.

Let $\mathcal{D}_u =\mathcal D((u^{-1}))$
be the superalgebra of Laurent series in $u^{-1}$ with values in $\mathcal{D}$. The algebra  $\mathcal{D}_u$ has a derivation $\du$ and $\mathcal{D}_u((\du^{-1}))$ is the superalgebra of pseudodifferential operators.

Let $G(\bs\La,\bs z)$ be a $k\times k$ matrix with entries in $\mathcal{D}_u[\du]\subset\mathcal{D}_u((\du^{-1}))$ 
 given by \[G(\bs\La,\bs z)=\hat \pi_k G(\bs z)=\left(\de_{a,b}(\du-z_a)-\sum_{i=1}^{m+n}\frac{x_{i,a}\dd_{i,b}}{u-\La_i}\right)_{a,b=1,\dots,k}.\] 

The matrix $G(\bs\La,\bs z)$ is a Manin matrix of parity $(1,\dots,1)$.
We want to expand $\cdet G(\bs\La,\bs z)$. In order to do that, we introduce some notation.

The superalgebra $\mathcal{D}_u((\du^{-1}))$ is topologically generated by $x_{i,a},\partial_{i,a},u^{\pm 1},\partial_u^{\pm 1}$. Define an ordering on the generators such that $x_{i,a}<\dd_{j,b}<u^{\pm 1}<\du^{\pm1}$, $i,j=1,\dots,m+n$, $a,b=1,\dots,k$, and $x_{i,a}<x_{j,b}$, $\dd_{i,a}<\dd_{j,b}$, if either $a<b$ or $a=b$ and $i<j$. 

Let $m$ be a monomial in the generators. Denote by $:\hspace{-3pt}m\hspace{-3pt}:$ the new monomial where all participating generators are multiplied in the increasing order and the sign is changed by the usual supercommutativity rule.
For example, 
\[
:\du^{-1}u^{-1}\dd_{1,1}x_{1,1}\dd_{m+1,2}x_{m+1,1}:=-x_{1,1}x_{m+1,1}\dd_{1,1}\dd_{m+1,2}u^{-1}\du^{-1}.\]
We call $:\hspace{-3pt}m\hspace{-3pt}:$ the {\it normal ordered monomial}.

Let \be\label{equation:definition of F}
F^{i}_{a,b}=\begin{cases}
-x_{i,a}\dd_{i,b}(u-\La_{i})^{-1}, &  i=1,\dots,m+n,\ a,b=1,\dots,k,\\
\du-z_a, &  a=1,\dots,k,\ b=a,\ i=0,\\
0,&{\rm otherwise}.
\end{cases}  
\ee Note that in all cases $F^{i}_{a,b}$ is even and normal ordered. In the expansion of $\cdet G(\bs\La,\bs z)$, every term will be given as a product of $F^{i}_{a,b}$.

Denote by $|S|$ the cardinality of a set $S$.

Let $\mathfrak a=\{1\leq a_1<\dots<a_{l}\leq k\}$ be a subset of $\{1,\dots,k\}$, where $l=|\mathfrak{a}|$.
Let $J(\mathfrak a)$ be the set of function $j:\ \{1,\dots,k\}\to\{0,1,\dots,m+n\}$  such that $j(a)=0$ if and only if $a\not\in\mathfrak{a}$ and such that for any $i\in\{1,\dots,m\}$, $|j^{-1}(i)|\leq 1$.

We have 
\[
|J(\mathfrak a)|=\sum_{s=0}^{|\mathfrak{a}|}\binom{l}{s}\binom{m}{s}\ s!\ n^{l-s}.
\]
For $j_1,j_2\in J(\mathfrak a)$, we write $j_1 \sim j_2$ if $|j_1^{-1}(i)|=|j_2^{-1}(i)|$ for all $i$.
Clearly, $\sim$ is an equivalence relation in $J(\mathfrak a)$. The cardinality of the equivalence class of $j\in J(\mathfrak a)$ is $l!/(\prod_{i=m+1}^{m+n}|j^{-1}(i)|!)$.

For $j\in J(\mathfrak a)$, $j^{-1}(\{1,\dots,m+n\})=\mathfrak a$. Therefore the symmetric group $\sS_{|\mathfrak{a}|}$ acts on the preimage $j^{-1}(\{1,\dots,m+n\})$. Given $j_1\sim j_2$, there exists a unique permutation $\si_{j_1,j_2}\in\sS_{|\mathfrak{a}|}$ such that $\si_{j_1,j_2}:j_2^{-1}(i)\rightarrow j_1^{-1}(i)$ is an increasing function for all $i=1,\dots,m+n$. Note that $j_1\circ \si_{j_1,j_2}=j_2$ on $\mathfrak a$. 

We also define
\begin{align*}
&\sgn(j_1,j_2)=(-1)^{\#}, \\ &\#=|\{ (s,s')\ {\rm such\ that}\ 1\leq s<s'\leq l,\ \si_{j_1,j_2}(s)<\si_{j_1,j_2}(s'), \ j_2(a_s)>m, \ j_2(a_{s'})>m\}|.
\end{align*}

Given $j_1, j_2\in J(\mathfrak{a})$, $j_1\sim j_2$, define the sign 
\[
c(j_1,j_2)=\sgn(j_1,j_2)\ \sgn(\si_{j_1,j_2})\,(-1)^{l}.
\]

For $j\in J(\mathfrak{a})$, set 
$$
\bs x_j=x_{j(a_1),a_1}x_{j(a_2),a_2}\dots x_{j(a_{l}),a_{l}}, \qquad 
\bs \dd_j=\dd_{j(a_1),a_1}\dd_{j(a_2),a_2}\dots \dd_{j(a_{l}),a_{l}}.
$$
Note that monomials $\bs x_j$ and $\bs \dd_j$ are normal ordered.

Now we are ready to state the main result of this section. 
\begin{prop}\label{prop:capelli}
The normal ordered expansion of the column determinant of $G(\bs\La,\bs z)$ is given by
\be\label{equation:capelli}
\cdet G(\bs\La,\bs z)=\hspace{-5pt}\sum_{\scriptstyle \mathfrak a\subset\{1,\dots,k\}}\sum_{\scriptstyle j_1,j_2\in J(\mathfrak a)\atop\scriptstyle j_1\sim j_2} c(j_1,j_2)\prod_{i=m+1}^{m+n}|j_2^{-1}(i)|!\ \bs x_{j_1}\bs \dd_{j_2}\prod_{i\in j_2(\mathfrak a)}(u-\La_i)^{-1}\prod_{a\not\in { \mathfrak a}}(\du-z_a).
\ee
\end{prop}
\begin{proof}  We first assume all generators are supercommutative and show equation \eqref{equation:capelli} holds. Then we show that the additional terms created by non-trivial supercommutation relations cancel in pairs and do not contribute to the expansion. 

Recall even elements $F^{i}_{a,b}$ given in \eqref{equation:definition of F}. We have the expansion 
\[
\cdet G(\bs\La,\bs z)=\sum_{\si\in\sS_k}\sum_{i_1,\dots,i_k=0}^{m+n}\sgn(\si)F^{i_1}_{\si(1),1}\dots F^{i_k}_{\si(k),k}.
\]
Now we want to normal order it.

\medskip 

Assume the supercommutators are all zero, $[u,\dd_u]=[x_{i,a},\dd_{i,a}]=0$.

For a nonzero term $\sgn(\si)F^{i_1}_{\si(1),1}\dots F^{i_k}_{\si(k),k}$, let $\mathfrak{a}=\{a,i_a\neq 0\}\subset \{1,\dots, k\}$. We write the set 
$\mathfrak a=\{a_1<\dots <a_l\}$. Then we can rewrite our sum as follows
\[
\cdet G(\bs\La,\bs z)=\sum_{l=0}^k\sum_{1\leq a_1<\dots<a_l\leq k}\sum_{\si\in\sS_{l}}\sum_{i_{1},\dots,i_{l}=1}^{m+n} \sgn(\si)F^{i_1}_{a_{\si(1)},a_1}\dots F^{i_{l}}_{a_{\si(l)},a_{l}}\prod_{a,\ a\neq a_1,\dots, a_l}(\dd-z_a).
\]  

We normal order the term corresponding to $a_1<\dots<a_{l},\si\in\sS_{l}, i_1,\dots,i_{l}$. Let $i_{\bar{1}}$ be the number of upper indices greater than $m$, $i_{\bar{1}}=|\{i_s>m,s=1,\dots,l\}|$. We have
\[
F^{i_1}_{a_{\si(1)},a_1}\dots F^{i_{l}}_{a_{\si(l)},a_l}=
(-1)^{l+i_{\bar{1}}(i_{\bar{1}}-1)/2}\ \frac{x_{i_1,a_{\si(1)}}\dots x_{i_l,a_{\si(l)}}\dd_{i_1,a_1}\dots \dd_{i_l,a_l}}{(u-\La_{i_1})\dots(u-\La_{i_l})}.
\] 

Note that monomial $\dd_{i_1,a_1}\dots \dd_{i_l,a_l}$ is normal ordered. We now observe some simplifications before ordering $x_{i_1,a_{\si(1)}}\dots x_{i_l,a_{\si(l)}}$.

Consider a term corresponding to $a_1<\dots<a_l$, $\si$, $i_1,\dots,i_l$.

Fix an $i\in \{1,\dots, m+n\}$. Let $\mathfrak b=\{s,\ i_s=i\}\subset \{1,\dots,l\}$. If $|\mathfrak b|=r>1$, then we have $r!$ terms which correspond to the same $a_1<\dots<a_l$, $i_1,\dots, i_l$, and permutations of the form $\tau\sigma$, where $\tau\in \sS_l$ permutes elements of $a_s, \  s\in\mathfrak b,$ and leaves others preserved.

If $i\leq m$, then after normal ordering all these $r!$ terms will produce the same monomial with different signs and cancel out. On the other hand, if $i> m$, then after normal ordering, all these $r!$ terms will produce the same monomial with the same sign and therefore can be combined.

Therefore, the summands in the expansion can be reparametrized by $\mathfrak a\subset\{1,\dots,k\}$ and
$j_1,j_2\in J(\mathfrak{a})$, $j_1\sim j_2$. The correspondence is given by 
\be\label{j def}
\mathfrak{a}=\{a_1<\dots<a_l\}, \qquad j_1(a_{\si(s)})=j_2(a_s)=i_s, \qquad s=1,\dots, l.
\ee
Note that $\sigma$ is not recovered from $\mathfrak a,j_1,j_2$. In fact, we have  
$|\{s,\ i_s=m+1\}|!\dots|\{s,\ i_s=m+n\}|!$ choices for $\sigma$, which all correspond to equal summands. 
We choose one permutation, namely $\si_{j_1,j_2}$, and multiply the corresponding term by $|j_1^{-1}(m+1)|!\dots|j_1^{-1}(m+n)|!$.

So we have
\begin{align*}
\cdet G(\bs\La,\bs z)=\hspace{-5pt}\sum_{\scriptstyle \mathfrak a\subset\{1,\dots,k\}}\sum_{\scriptstyle j_1,j_2\in J(\mathfrak a)\atop\scriptstyle j_1\sim j_2}\ \sgn(\si_{j_1,j_2})\ (-1)^{|j_1|_{\bar 1}(|j_1|_{\bar 1}-1)/2+|\mathfrak a|}\prod_{i=m+1}^{m+n}|j_2^{-1}(i)|! \\
\ \times\frac{x_{j_2(a_1),a_{\si_{j_1,j_2}(1)}}\dots x_{j_2(a_{|\mathfrak a|}),a_{\si_{j_1,j_2}(|\mathfrak a|)}}\dd_{j_2(a_1),a_1}\dots \dd_{j_2(a_{|\mathfrak a|}),a_{|\mathfrak a|}}}{(u-\La_{i_1})\dots(u-\La_{i_{|\mathfrak a|}})}\ \prod_{a\not\in { \mathfrak a}}(\du-z_a)\ ,
\end{align*}
where we denoted by $|j_1|_{\bar 1}$ the cardinality of $j_1^{-1}(\{m+1,\dots,m+n\})$. 
We rewrite the first indices of  $x_{i,a}$ variables through $j_1$, using \eqref{j def}, and then we normal order them, getting the additional sign and arriving at \eqref{equation:capelli}.

\medskip

Now we proceed to the non-commutative setting. We call the additional terms "quantum corrections" and show that they cancel in pairs. 

We normal order monomials from right to left. The induction is based on the number of $F^{i}_{a,b}$ on the right which have been normal ordered. Namely we prove
 \be \label{ordered} \cdet G(\bs\La,\bs z)=\sum_{\si\in\sS_k}\sum_{i_1,\dots,i_k=0}^{m+n}\sgn(\si)\, F^{i_1}_{\si(1),1}\dots :F^{i_{a}}_{\si(a),a}\dots F^{i_k}_{\si(k),k}:\ee
by induction on $a$.

The basis $a=k$ of induction is a tautology. 
We show the step of induction from $a=a_0$ to $a=a_{0}-1$.
 
We use the following simple formula:
$$
F^{i_1}_{a_1,b_1}F^{i_2}_{a_2,b_2}=:F^{i_1}_{a_1,b_1}F^{i_2}_{a_2,b_2}:-
\begin{cases}
F^{i_2}_{a_2,b_2}(u-\La_{i_2})^{-1}, & i_1=0,\ a_1=b_1,\ i_2\neq 0, \\
\delta_{i_1,i_2}\delta_{b_1,a_2} F^{i_1}_{a_1,b_2} (u-\La_{i_2})^{-1}, & i_1\neq 0, \\
0, & {\rm otherwise}.
\end{cases}
$$
Consider a nonzero term $\sgn(\si) F^{i_{a_0-1}}_{\si(a_0-1),a_0-1}:F^{i_{a_0}}_{\si(a_0),a_0}\dots F^{i_k}_{\si(k),k}:$. 
Then $i_a=0$ implies $\si(a)=a$.

We have two cases: $i_{a_0-1}\neq 0$ and $i_{a_0-1}= 0$.

Let $i_{a_0-1}\neq 0$. Then $F^{i_{a_0-1}}_{\si(a_0-1),a_0-1}$ creates at most one quantum correction.
Namely, if there exists $b\in\{a_0,\dots,k\}$ such that $i_{a_0-1}=i_b$, and $a_0-1=\si(b)$, then such $b$ is unique and 
\begin{align}\label{equation:nomoral product 1}\nonumber
\sgn(\si)\, F^{i_{a_0-1}}_{\si(a_0-1),a_0-1}:&F^{i_{a_0}}_{\si(a_0),a_0}\dots F^{i_b}_{\si(b),b} \dots F^{i_k}_{\si(k),k}:\\ \nonumber
=&\sgn(\si) :F^{i_{a_0-1}}_{\si(a_0-1),a_0-1}F^{i_{a_0}}_{\si(a_0),a_0}\dots F^{i_b}_{\si(b),b} \dots F^{i_k}_{\si(k),k}:\\ 
&-\sgn(\si) \frac{1}{u-\La_{i_b}}:F^{i_{a_0}}_{\si(a_0),a_0}\dots F^{i_b}_{\si(a_0-1),b} \dots F^{i_k}_{\si(k),k}:.
\end{align}
If such $b$ does not exist then there is no quantum correction (the second term on the right hand side is absent).

Let $i_{a_0-1}= 0$. Then we possibly have many quantum corrections:
\begin{align}\label{equation:nomoral product 2}\nonumber
\sgn(\si)F^{0}_{a_0-1,a_0-1}:F^{i_{a_0}}_{\si(a_0),a_0} \dots F^{i_k}_{\si(k),k}:=\sgn(\si) :F^0_{a_0-1,a_0-1}F^{i_{a_0}}_{\si(a_0),a_0} \dots F^{i_k}_{\si(k),k}:\\ 
-\sgn(\si) \sum_{\scriptstyle  a=a_0\atop i_a\neq0}^{k}\frac{1}{u-\La_{i_a}}:F^{i_{a_0}}_{\si(a_0),a_0} \dots F^{i_k}_{\si(k),k}:.
\end{align}

The quantum correction in \eqref{equation:nomoral product 1} corresponding to the term labeled by $\si, i_1,\dots,i_k$ in \eqref{ordered} cancels with the quantum correction corresponding to $a=b$ summand in \eqref{equation:nomoral product 2} applied to the term in \eqref{ordered} labeled by  $\si(a_0-1,b)$, $\{i_1,\dots,i_{a_0-2},0,i_{a_0},\dots, i_k\}$. 
This proves the induction step.

The statement of induction with $a=1$ proves the proposition.
\end{proof}

\subsection{Another identity of Capelli type}\label{subsec:another capelli} Let $\mathcal{D}_v=\mathcal{D}((v^{-1}))$, be the superalgebra of Laurent series in $v^{-1}$ with coefficients in $\mathcal{D}$. The superalgebra $\mathcal{D}_v$ has a derivation $\dv$ and we consider the superalgebra of pseudodifferential operators $\mathcal{D}_v((\dv^{-1}))$.

Let $B(\bs z,\bs\La)$ be a $(m+n)\times(m+n)$ matrix with entries in $\mathcal D_v[\dv]\subset \mathcal D_v((\dv^{-1}))$ given by 
\[
B(\bs z,\bs\La)=\hat\pi_{m|n}B(\bs \La)=\left(\de_{ij}(\dv-\La_i)-\sum_{a=1}^{k}\frac{(-1)^{\bar i}x_{i,a}\dd_{j,a}}{v-z_a}\right)_{i,j=1,\dots,m+n}.
\] The matrix $B(\bs z,\bs\La)$ is a Manin matrix of standard parity.

Let $\hat{B}(\bs z,\bs\La)$ be a $(m+n+k)\times(m+n+k)$ matrix given by
\be\label{equation: one side of capelli matrix}
\hat{B}(\bs z,\bs\La)=\bpm
v-Z	&D^t\\
SX	&\dv-\La
\epm
\ee where the submatrices are $Z=\diag(z_1,\dots,z_k)$, $\La=\diag(\La_1,\dots,\La_{m+n})$, $D=\big(\dd_{i,a}\big)_{i=1,\dots,m+n}^{a=1,\dots,k}$, $X=\big(x_{i,a}\big)_{i=1,\dots,m+n}^{a=1,\dots,k}$, $S=\diag(\underbrace{1,\dots,1}_m,\underbrace{-1,\dots,-1}_n)$, and $D^t$ is the transpose of $D$. In particular, $SX=\big((-1)^{\bar{i}}x_{i,a}\big)_{i=1,\dots,m+n}^{a=1,\dots,k}$.

Let $\mathcal{D}_v((\dv^{-1}))((w))$ be the superalgebra of Laurent series in $w$ with coefficients in $\mathcal{D}_v((\dv^{-1}))$. Define the homomorphism of superalgebras
\begin{align}\label{phi hat}
&\hat\Phi: \mathcal{D}_v((\dv^{-1}))\to  \mathcal{D}_v((\dv^{-1}))((w)),\notag\\
&v\mapsto v+w^{-1},\ \dv\mapsto \dv+w^{-1},\ {\rm and}\  g\mapsto g,\ g\in\mathcal D.\end{align}
Note that in our convention we first expand in positive powers of $w$ then in powers of $\dv^{-1}$ and then in powers of $v^{-1}$, cf. \eqref{equation:Phi map}.
As a result, if a series is in the image of $\hat\Phi$, then it belongs to $\mathcal D[v,\dv]((w))$, in other words, a coefficient of $w^k$ is always a polynomial in $\dv$ and $v$ for any $k\in\Z$. 

The map $\hat \Phi$ is a composition of map $\Phi$, see \eqref{equation:Phi map} and of the shift homomorphism $v\to v+w^{-1}$. Therefore, $\hat \Phi$ is a well-defined injective homomorphism.

Then, it is straightforward to check the following statement.

\begin{lem}
The matrix $\hat{B}(\bs z,\bs\La)$ is an affine-like Manin matrix of parity $\hat{\bs s}_0=(\underbrace{1,\dots,1}_{k+m},\underbrace{-1,\dots,-1}_n)$ with the map $\hat\Phi$. \qed
\end{lem}

We would like to expand and normal order the Berezinian of $B(z,\bs \La)$. 
However, it is sufficient to expand and normal order Berezinian of $\hat{B}(\bs z,\bs\La)$.
Indeed, by Proposition \ref{prop:affine-like properties}, we have 
\be\label{equation:ber for gl_m|n}
{\rm Ber}^{\hat{\bs s}_0}\, \hat{B}(\bs z,\bs\La)=(v-z_1)\dots(v-z_{k})\ \ber B(\bs z, \bs\La),
\ee
cf. Corollary 2.2 of \cite{MTV5}.

The expansion of the Berezinian of $\hat{B}(\bs z,\bs\La)$ is given by the following proposition.
\begin{prop}\label{prop:ber for gl_m|n}
We have 
\begin{align}
{\rm Ber}^{\hat{\bs s}_0}\,\hat{B}(\bs z,\bs\La)=\hspace{-5pt}\sum_{\scriptstyle \mathfrak{a}\subset\{1,\dots,k\}}\sum_{\scriptstyle j_1,j_2\in J(\mathfrak{a})\atop\scriptstyle j_1\sim j_2} c(j_1,j_2)\prod_{i=m+1}^{m+n}|j_2^{-1}(i)|!&\ \bs x_{j_1}\bs \dd_{j_2}\prod_{a\not\in \mathfrak{a}}(v-z_a)\prod_{i\in j_1(\mathfrak{a})}(\dv-\La_i)^{-1}\notag\\ 
&\times\frac{(\dv-\La_1)\dots(\dv-\La_m)}{(\dv-\La_{m+1})\dots(\dv-\La_{m+n})} \label{equation:ber for gl_m|n expansion}\ .
\end{align}
\end{prop}
\begin{proof} Let $\si\in\sS_{m+n+k}$ be defined by $\si^{-1}(a)=m+n+a$, $a=1,\dots,k$, and $\si^{-1}(k+i)=i$, $i=1,\dots,m+n$. 
Then
$$\si(\hat{B}(\bs z,\bs\La))=\bpm \dv-\La&SX\\D^t&v-Z\epm.$$ 
	
	The matrix $\si(\hat{B}(\bs z,\bs\La))$ is an affine-like Manin matrix of parity $\bs s=(\underbrace{1,\dots,1}_m,\underbrace{-1,\dots,-1}_n,\underbrace{1,\dots,1}_k)$ with the map $\hat\Phi$. By Proposition \ref{prop:affine-like properties}, we have 
\[
{\rm Ber}^{\hat{\bs s}_0}\hat{B}(\bs z,\bs\La)=\bers\si(\hat{B}(\bs z,\bs\La)).
\]

Using Proposition \ref{prop:affine-like properties}  once again (we use $r=m+n$), we further see 
\[
\bers\si(\hat{B}(\bs z,\bs\La))=\frac{(\dv-\La_1)\dots(\dv-\La_m)}{(\dv-\La_{m+1})\dots(\dv-\La_{m+n})}\ \cdet B'(\bs z,\bs\La)\ ,
\] where  $B'(\bs z,\bs\La)$ is an even matrix given by
\[B'(\bs z,\bs\La)=\Big(\de_{a,b}(v-z_a)-\sum_{i=1}^{m+n}\frac{(-1)^{\bar i}\dd_{i,a}x_{i,b}}{\dv-\La_i}\Big)_{a,b=1,\dots,k}\ .\] 

Next we move the factor $\frac{(\dv-\La_1)\dots(\dv-\La_m)}{(\dv-\La_{m+1})\dots(\dv-\La_{m+n})}$ to the right of the column determinant. Note that for $i\in\{1,\dots,m\},$ $a\in\{1,\dots,k\}$, we have \[(\dv-\La_i)(v-z_a-\frac{\dd_{i,a}x_{i,a}}{\dv-\La_i})=(v-z_a-\frac{x_{i,a}\dd_{i,a}}{\dv-\La_i})(\dv-\La_i)\ .\]
Similarly, for  $i\in\{m+1,\dots,m+n\},$ $a\in\{1,\dots,k\}$, 
\[
\frac{1}{(\dv-\La_i)}(v-z_a+\frac{\dd_{i,a}x_{i,a}}{\dv-\La_i})=(v-z_a-\frac{x_{i,a}\dd_{i,a}}{\dv-\La_i})\frac{1}{(\dv-\La_i)}\ .
\] Therefore, we have
\begin{align*}
{\rm Ber}^{\hat{\bs s}_0}\ \hat{B}(\bs z,\bs\La)=\cdet \Big(\de_{a,b}(v-z_a)-\sum_{i=1}^{m+n}&\frac{x_{i,b}\dd_{i,a}}{\dv-\La_i}\Big)_{a,b=1,\dots,k}\hspace{-5pt}\times\frac{(\dv-\La_1)\dots(\dv-\La_m)}{(\dv-\La_{m+1})\dots(\dv-\La_{m+n})}\ .
\end{align*}
Finally, the expansion of the above column determinant is done by a computation similar to the one in Proposition \ref{prop:capelli}. 
\end{proof}

Theorem \ref{thm:duality between bethe algebras} follows from Propositions \ref{prop:capelli} and \ref{prop:ber for gl_m|n}.

\medskip

We remark that the $k\times k$ column determinant $\cdet G(\bs\La,\bs z)$ in Proposition \ref{prop:capelli} is also essentially a Berezinian of an $(m+n+k)\times(m+n+k)$ matrix.
Namely, let $\hat{G}(\bs\La,\bs z)$ be a $(m+n+k)\times(m+n+k)$ matrix given by 
\[
\hat{G}(\bs\La,\bs z)=\bpm
u-\La	&D\\
X^t	&\du-Z
\epm.
\]  Then $\hat{G}(\bs\La,\bs z)$ is an affine-like Manin matrix of parity $\bs s=(\underbrace{1,\dots,1}_m,\underbrace{-1,\dots,-1}_n,\underbrace{1,\dots,1}_k)$ with the same homomorphism of superalgebras $\hat{\Phi}$, see \eqref{phi hat}. By Proposition \ref{prop:affine-like properties}, the Berezinian of parity $\bs s$ of $\hat{G}(\bs\La,\bs z)$ is given by 
\be\label{equation:ber for gl_k}
\bers\hat{G}(\bs\La,\bs z)=\frac{(u-\La_1)\dots(u-\La_m)}{(u-\La_{m+1})\dots(u-\La_{m+n})}\ \cdet G(\bs\La,\bs z).
\ee


\begin{thebibliography}{Gor004}
	\bibitem[BBK]{BBK}	
	L. Banchi, D. Burgarth, and M. J. Kastoryano,
	\emph{Driven Quantum Dynamics: Will It Blend?},	Phys. Rev. X, {\bf 7} (2017), 041015
	
	\bibitem[CFR]{CFR} A. Chervov, G. Falqui, and V. Rubtsov, {\it Algebraic properties of Manin matrices}, I. Adv. in Appl. Math. {\bf 43} (2009), no.\, 3, 
	239--315
	
	
	\bibitem[HMVY]{HMVY}
	C. Huang, E. Mukhin, B. Vicedo, and C. Young,
	\emph{The solution of $\gl_{M|N}$ Bethe Ansatz equation and rational pseudodifferential operators}, arXiv:1809.01279, 28 pp
	
	\bibitem[MM]{MM} A. Molev and E. Mukhin, {\it Invariants of the vacuum module associated with the Lie superalgebra 
$\gl(1|1)$}, J. Phys. A {\bf 48} (2015), no.\, 31, 314001, 20 pp
	
	
	\bibitem[MR]{MR}
	A. I. Molev and E. Ragoucy,
	\emph{The MacMahon Master Theorem for right quantum superalgebras and higher Sugawara operators for $\widehat{\mathfrak{gl}}(m|n)$},
	Moscow Math. J., {\bf 14} (2014), no.\,1, 83--119
	
	\bibitem[MTV1]{MTV1}
	E. Mukhin, V. Tarasov, and A. Varchenko,
	\emph{Bethe eigenvectors of higher transfer matrices}, J. Stat. Mech. Theory Exp. (2006), no.\, 8, P08002, 44 pp
		
	\bibitem[MTV2]{MTV2}
	E. Mukhin, V. Tarasov, and A. Varchenko,
	\emph{Bispectral and $(\gl_N,\gl_M)$ dualities}, Funct. Anal. Other Math. {\bf 1} (2006), no.\,1, 47--69
	
	\bibitem[MTV3]{MTV3}
	E. Mukhin, V. Tarasov, and A. Varchenko,
	\emph{Bispectral and $(\gl_N,\gl_M)$ dualities, discrete versus differential}, Adv. Math. {\bf 218} (2008), no.\,1, 216--265
	
	\bibitem[MTV4]{MTV4}
	E. Mukhin, V. Tarasov, and A. Varchenko, {\it Spaces of quasi-exponentials and representations of 
$\gl_N$}, J. Phys. A {\bf 41} (2008), no.\, 19, 28 pp
	
	
	\bibitem[MTV5]{MTV5}
	E. Mukhin, V. Tarasov, and A. Varchenko,
	\emph{A generalization of the Capelli identity}, Algebra, arithmetic, and geometry: in honor of Yu. I. Manin. Vol. II, Progr. Math., {\bf 270}, Birkhäuser Boston, Inc., Boston, MA, 2009, pp. 383--398 
	
	\bibitem[MVY]{MVY}
	E. Mukhin, B. Vicedo, and C. Young, 
	\emph{Gaudin models for $\gl(m|n)$}, J. Math. Phys. {\bf 56} (2015), no.\,5, 051704, 30 pp
	
	\bibitem[T]{T}
	D. Talalaev, \emph{Quantization of the Gaudin System },  arXiv:hep-th/0404153, 9pp
	
\end{thebibliography}
\end{document}